\documentclass[a4paper,twoside,12pt]{article}

\usepackage{amssymb,amsmath,amsthm,latexsym}
\usepackage{amsfonts}
\usepackage{graphicx,caption,subcaption}
\usepackage[pdftex,bookmarks,colorlinks=false]{hyperref}
\usepackage[hmargin=1.2in,vmargin=1.2in]{geometry}
\usepackage[mathscr]{euscript}

\usepackage[auth-lg,affil-sl]{authblk}
\allowdisplaybreaks

\usepackage{float}
\usepackage{calc,tikz}
\usetikzlibrary{arrows,decorations.markings}
\tikzstyle{vertex}=[circle, draw, inner sep=1pt, minimum size=4pt]
\newcommand{\vertex}{\node[vertex]}

\usepackage{multirow}
\usepackage{hhline}
\usepackage{booktabs}
\usepackage{longtable}

\newcommand{\noi}{\noindent}
\newcommand{\N}{\mathbb{N}}

\newtheorem{theorem}{Theorem}[section]
\newtheorem{definition}{Definition}[section]
\newtheorem{lemma}[theorem]{Lemma}
\newtheorem{proposition}[theorem]{Proposition}
\newtheorem{corollary}[theorem]{Corollary}

\newtheorem{problem}{Problem}

\title{\textbf{\sc Some Properties of Fibonacci-Sum Set-Graphs}}

\author{Eunice Gogo Mphako-Banda}
\affil{\small School of Mathematical Sciences\\ University of Witswatersrand \\Johannesburg, South Africa.\\{\tt eunice.mphako-banda@wits.ac.za}}

\author{Johan Kok}
\affil{\small Centre for Studies in Discrete Mathematics\\ Vidya Academy of Science \& Technology \\Thalakkottukara, Thrissur, Kerala, India.\\{\tt $^\ast$kokkiek2@tshwane.gov.za, }}

\author{Sudev Naduvath\footnote{Corresponding Author}}
\affil{\small Centre for Studies in Discrete Mathematics\\ Vidya Academy of Science \& Technology \\Thalakkottukara, Thrissur, Kerala, India.\\{\tt sudevnk@gmail.com}}

\date{}

\begin{document}
\maketitle

\begin{abstract}
\noindent In this paper we study some properties of Fibonacci-sum set-graphs. The aforesaid graphs are an extension of the notion of Fibonacci-sum graphs to the notion of set-graphs. The colouring of Fibonacci-sum graphs (not set-graphs) is also discussed. A number of challenging research problems are posed in the conclusion.
\end{abstract}

\vspace{0.35cm}


\section{Introduction}

For general notation and concepts in graphs and digraphs see \cite{BM1,FH,DBW}. Unless stated otherwise, all graphs will be finite connected graphs with multiple edges and loops are permitted. The number of multiple (or single) edges between vertices $u$ and $v$ is denoted by $\epsilon(u,v)$. Therefore, $\epsilon(u,v) \geq 0$. Note that $\epsilon(u,v)=0$ implies that the vertices $v,u$ are non-adjacent. For a vertex $v$ its number of loops is denoted, $l(v)$. The \textit{open neighbourhood} of a vertex $v$, denoted by $N(v)$, is the set of vertices which are adjacent to $v$. The degree of a vertex $v \in V(G)$ is denoted $d_G(v)$ or when the context is clear, simply as $d(v)$.  The degree of a vertex $v$ is given by $d(v)= 2l(v)+\epsilon(v,u)$, where $u\in N(v)$. 

Recall that the \textit{sequence of Fibonacci numbers} $\mathcal{F} =\{f_n\}_{n\geq 0},\ n\in \N_0$ is defined recursively as $f_0=0$, $f_1=1$ and $f_n=f_{n-1}+f_{n-2}$. As defined in \cite{FKM}, a \textit{Fibonacci-sum graph} is defined for a finite set of the first $n$ consecutive positive integers $\{1,2,3,\ldots,n\}$ as $G^F_n$ with $V(G^F_n)=\{v_i:1\leq i\leq n\}$ and $E(G^F_n)=\{v_iv_j:i\neq j, i+j \in \mathcal{F}\}$. 

In this paper, we introduce the notion of a new class of graphs, namely the Fibonacci-sum set-graphs and study some fundamental structural characteristics of this graph class. Fibonacci-sum set-graphs are an extension of the notion of Fibonacci-sum graphs to the notion of set-graphs.

\section{Derivative Set-graphs}

\noi The notion of a set-graph was introduced in \cite{KCSS} as explained below.

\begin{definition}\label{Defn-2.1}{\rm 
\cite{KCSS} Let $A^{(n)} = \{a_1,a_2,a_3,\dots , a_n\}$, $ n \in \N$ be a non-empty set and the $i$-th $s$-element subset of $A^{(n)}$ be denoted by $A^{(n)}_{s,i}$. Now, consider $\mathcal S = \{A^{(n)}_{s,i}: A^{(n)}_{s,i} \subseteq A^{(n)}, A^{(n)}_{s,i} \neq \emptyset \}$. The \textit{set-graph} corresponding to set $A^{(n)}$, denoted $G_{A^{(n)}}$, is defined to be the graph with $V(G_{A^{(n)}}) = \{v_{s,i}: A^{(n)}_{s,i} \in \mathcal S\}$ and $E(G_{A^{(n)}}) = \{v_{s,i}v_{t,j}: A^{(n)}_{s,i} \cap A^{(n)}_{t,j} \neq \emptyset\}$, where $s\neq t$ or $i\neq j$.
}\end{definition}

\noi Note that the definition of vertices implies, $v_{s,i} \mapsto A^{(n)}_{s,i} \in \mathcal S$.

\begin{figure}[h!]
\centering
\begin{tikzpicture}[auto,node distance=1.75cm,
thick,main node/.style={circle,draw,font=\sffamily\Large\bfseries}]
\vertex (v1) at (0:3) [fill,label=right:$v_{1,1}$]{};
\vertex (v2) at (308.57:3) [fill,label=right:$v_{1,2}$]{};
\vertex (v3) at (257.14:3) [fill,label=below:$v_{1,3}$]{};
\vertex (v4) at (205.71:3) [fill,label=below:$v_{2,1}$]{};
\vertex (v5) at (154.28:3) [fill,label=left:$v_{2,2}$]{};
\vertex (v6) at (102.85:3) [fill,label=left:$v_{2,3}$]{};
\vertex (v7) at (51.42:3) [fill,label=above:$v_{3,1}$]{};
\path 
(v1) edge (v4)
(v1) edge (v5)
(v1) edge (v7)
(v2) edge (v4)
(v2) edge (v6)
(v2) edge (v7)
(v3) edge (v5)
(v3) edge (v6)
(v3) edge (v7)
(v4) edge (v5)
(v4) edge (v6)
(v4) edge (v7)
(v5) edge (v6)
(v5) edge (v7)
(v6) edge (v7)
;
\end{tikzpicture}
\caption{\small The set-graph $G_{A^{(3)}}$.}\label{fig:fig-1}
\end{figure}
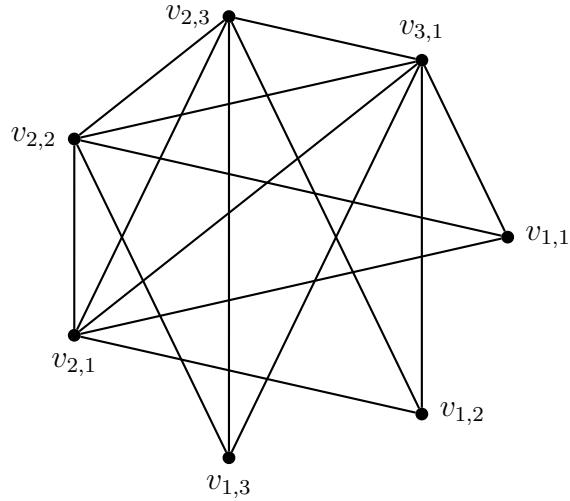 

\subsection{Fibonacci-sum set-graph}

Derivative set-graphs are obtained by considering well-defined sets of positive integers together with number theoretical operators or conditions between the integers or the vertices corresponding to the subsets. For colouring related results in respect of derivative set-graphs see \cite{KS1}. 

\noi The notion of a Fibonacci-sum set-graph can be introduced as follows.

\begin{definition}\label{Defn-2.2}{\rm
Let $A^{(n)} = \{1,2,3,\ldots, n\},\ n \in \N$ be a non-empty set and the $i$-th $s$-element subset of $A^{(n)}$ be denoted by $A^{(n)}_{s,i}$. Now, consider $\mathcal{S}=\{A^{(n)}_{s,i}: A^{(n)}_{s,i} \subseteq A^{(n)}, A^{(n)}_{s,i} \neq \emptyset \}$. The \textit{Fibonacci-sum set-graph} corresponding to set $A^{(n)}$, denoted $G^F_{A^{(n)}}$, is defined to be the graph with $V(G^F_{A^{(n)}})=\{v_{s,i}: A^{(n)}_{s,i} \in \mathcal{S}\}$ and $E(G^F_{A^{(n)}}) = \{v_{s,i}v_{t,j}: \forall (i',j'), i'\in A^{(n)}_{s,i}, j'\in A^{(n)}_{t,j},i' \neq j' \mbox{and the sum}\ i'+ j' \in \mathcal{F}\}$.
}\end{definition}

Since $A^{(n)}_{s,i}$ and $A^{(n)}_{t,j}$ are not necessarily distinct, it follows that loops are permitted. Figure \ref{fig:fig-2} depicts the Fibonacci-sum set-graph $G^F_{A^{(3)}}$. Significant differences exist between the Fibonacci-sum graph $G^F_3 \cong P_3$, the set-graph $G_{A^{(3)}}$ and the Fibonacci-sum set-graph $G^F_{A^{(3)}}$. Evidently, $\nu(G^F_n)\leq \nu(G_{A^{(n)}})=\nu(G^F_{A^{(n)}})$ and $\varepsilon(G^F_n)\leq \varepsilon(G_{A^{(n)}})\leq\varepsilon(G^F_{A^{(n)}})$ with equality throughout only for $n=1$.

\begin{figure}[h!]
\centering
\begin{tikzpicture}[auto,node distance=1.75cm,
thick,main node/.style={circle,draw,font=\sffamily\Large\bfseries}]
\vertex (v1) at (51.42:4) [fill,label=right:$v_{3,1}$]{};
\vertex (v2) at (0:4) [fill,label=right:$v_{1,1}$]{};
\vertex (v3) at (308.57:4) [fill,label=below:$v_{1,2}$]{};
\vertex (v4) at (257.14:4) [fill,label=below:$v_{1,3}$]{};
\vertex (v5) at (205.71:4) [fill,label=left:$v_{2,1}$]{};
\vertex (v6) at (154.28:4) [fill,label=left:$v_{2,2}$]{};
\vertex (v7) at (102.85:4) [fill,label=above:$v_{2,3}$]{};
\path 
(v1) edge [bend left=10] (v2)
(v1) edge [bend right=10] (v2)
(v1) edge [bend left=10] (v3)
(v1) edge [bend right=10] (v3)
(v1) edge (v4)
(v1) edge [bend left=10] (v5)
(v1) edge [bend left=17.5] (v5)
(v1) edge (v5)
(v1) edge [bend right=10] (v5)
(v1) edge [bend right=10] (v6)
(v1) edge (v6)
(v1) edge [bend left=10] (v6)
(v1) edge (v7)
(v1) edge [bend left=10] (v7)
(v1) edge [bend right=10] (v7)
(v1) edge [loop right,in=50,out=100,looseness=55] (v1)
(v1) edge [loop right,in=30,out=120,looseness=55] (v1)
(v2) edge (v3)
(v2) edge [bend left=10]  (v5)
(v2) edge [bend right=10] (v5)
(v2) edge (v6)
(v2) edge (v7)
(v3) edge (v4)
(v3) edge (v5)
(v3) edge [bend left=10] (v6)
(v3) edge [bend right=10] (v6)
(v3) edge (v7)
(v4) edge (v5)
(v4) edge (v7)
(v5) edge [bend left=10] (v6)
(v5) edge (v6)
(v5) edge [bend right=10] (v6)
(v5) edge [bend left=10] (v7)
(v5) edge (v7)
(v5) edge [loop left,in=160,out=225,looseness=55] (v5)

(v6) edge [bend left=10] (v7)
(v6) edge [bend right=10] (v7)
(v7) edge [loop left,in=45,out=140,looseness=55] (v7)
;
\end{tikzpicture}
\caption{Fibonacci-sum set-graph $G^F_{A^{(3)}}$.}\label{fig:fig-2}
\end{figure}
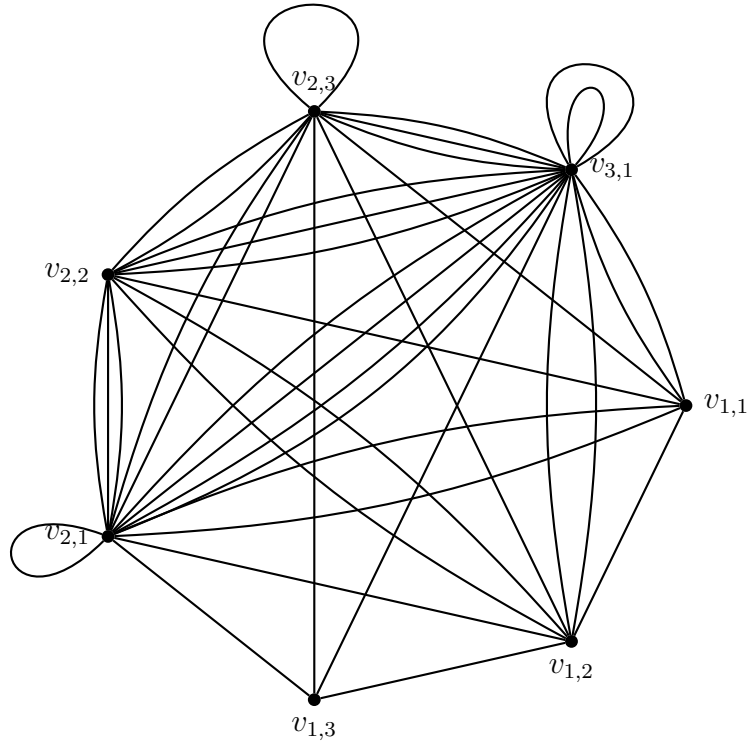 

\begin{theorem}\label{Thm-2.1}
For $n\in \N$, the Fibonacci-sum set-graph has no pendant vertices.
\end{theorem}
\begin{proof}
The result can easily be verified for $n=1,2$. Note that an isolated vertex is not a pendant vertex. Now assume that it holds for $n=k\geq 3$. The Fibonacci-sum set-graph is, at least partially, constructed as follows:

\textit{Step 1:} Take two copies of $G^F_{A^{(n)}}$ and label them $G_1$ and $G_2$. Define a map $g:v_{s,i}\mapsto v'_{s,i}$, where $v_{s,i}\in V(G_1)$, $v'_{s,i} \in V(G_2)$ and $A^{{(n)}'}_{s,i}= A^{(n)}_{s,i} \cup \{k+1\}$.

\textit{Step 2:} Add all edges between $G_1$ and $G_2$ corresponding to Definition \ref{Defn-2.2}. Hence, all loops at vertex $v_{s,i}$, if any, correspond to multiple edges between vertices $v_{s,i}$ and $v'_{s,i}$. Furthermore, since each $v_{s,i} \in V(G_1)$ has at least two neighbours, say $v_{t,j}$, $v_{m,q}$, in $G_1$, it follows that $v_{s,i}$ is adjacent to at least $v'_{t,j}$, $v'_{m,q}$ in $G_2$. So, connectivity is ensured with no pendant vertices resulting thus far.

\textit{Step 3(a):} Consider the vertex $v_{1,k+1}$ corresponding to the subset $\{k+1\}$. If $k+1$ is a Fibonacci number, say $f_\ell$, then $f_{\ell-1}$ exists in the finite set $A^{(n)}$ implying that at least two distinct subsets exist which contains $f_{\ell-1}$. Hence, $d(v_{1,k+1})\geq 2$ in $G^F_{A^{(n)}}$. Hence, the result holds for all $n\in \N$ and $n$ itself is a Fibonacci number.

\textit{Step 3(b):} If $k+1$ is not a Fibonacci number, then a largest Fibonacci number, say $f_ell$ belongs to $A^{(n)}$. Also, $f_{\ell+1} \notin A^{(n)}$. However, $\ell' = f_{\ell+1}-(k+1) \in A^{(n)}$. Similarly, at least two distinct subsets exist, each containing $\ell'$. Since $\ell' + (k+1) = f_{\ell+1}$, vertex $v_{1,k+1}$ has $d(v_{1,k+1})\geq 2$ in $G^F_{A^{(n)}}$ as well. Hence, the result holds for all $n\in \N$ and $n$ itself is a non-Fibonacci integer.

Hence, by mathematical induction together with the well-ordering property of the set of positive integers, the result follows for all $n\in \N$.
\end{proof}

The following two corollaries are the direct consequences of Theorem \ref{Thm-2.1}. 

\begin{corollary}\label{Cor-2.2}
For $n\geq 1$, the Fibonacci-sum set-graph is connected.
\end{corollary}

\begin{corollary}\label{Cor-2.3}
For $n\geq 1$, the Fibonacci-sum set-graph is Hamiltonian.
\end{corollary}

\begin{lemma}\label{Lem-2.4}
For $n\geq 1$, the corresponding Fibonacci-sum set-graph has a unique vertex $v$ for which $l(v)=\varepsilon(G^F_n)$.
\end{lemma}
\begin{proof}
For $n=1$, it follows that $G^F_1 \cong K_1 \cong G^F_{A^{(1)}}$. Therefore, in $G^F_{A^{(1)}}$, $l(v_{1,1})=0=\varepsilon(G^F_1)$. Clearly, the vertex $v_{1,1}$ is unique. Similarly, for $n=2$, the Fibonacci-sum graph is $G^F_2 =P_2$. The Fibonacci-sum set-graph has vertices $v_{1,1} =\{1\}$, $v_{1,2} = \{2\}$, $v_{2,1}=\{1,2\}$ and therefore has edges $v_{1,1}v_{1,2}$, $v_{1,1}v_{2,1}$, $v_{1,2}v_{2,1}$ and loop $v_{2,1}v_{2,1}$. Clearly, $l(v_{2,1})= 1 =\varepsilon(G^F_2)$ and $v_{2,1}$ is unique, since $A^{(2)}_{1,1}\subset A^{(2)}_{2,1}=A^{(2)}$ and $A^{(2)}_{1,2}\subset A^{(2)}_{2,1}=A^{(2)}$. Since $A^{(n)}_{s,i}\subset A^{(n)}\in \mathcal{S}$ and $v_{2^n-1,1}\mapsto A^{(n)}$ is unique, by mathematical induction, it follows that Definition \ref{Defn-2.1} and Definition \ref{Defn-2.2} imply $l(v_{2^n-1,1}) =\varepsilon(G^F_n)$.
\end{proof}

The next result is a direct derivative of Corollary 16 in \cite{AGL} read together with Lemma \ref{Lem-2.4}. Corollary 16 states that for $n \geq 1$ and $k \geq 2$ satisfying, $f_k\leq n \leq f_{k+1}$ then the number of edges of the Fibonacci-sum graph, $G^F_n$ is
\begin{equation*} 
\varepsilon(G^F_n) = 
\begin{cases}
n+ \frac{f_k+1}{2}- \frac{\lfloor\frac{4(k+1)}{3}\rfloor}{2}, &\text {if $n \leq \frac{f_{k+2}}{2}$},\\
2n +\frac{f_k+1}{2} - \frac{\lfloor\frac{4(k+1)}{3}\rfloor}{2} - \lceil \frac{f_{k+2}-1}{2}\rceil, & \text {if $n > \frac{f_{k+2}}{2}$}.
\end{cases}
\end{equation*} 
\begin{theorem}
Let $n \geq 1$ and $k \geq 2$ be integers satisfying the inequality $f_k\leq n \leq f_{k+1}$. For $A^{(n)} = \{1,2,3,\dots, n\}$ the vertex $v_{2^n-1,1}$ of the Fibonacci-sum set-graph, $G^F_{A^{(n)}}$ has maximum loop number, 
\begin{equation*} 
l(v_{2^n-1,1}) = 
\begin{cases}
n+ \frac{f_k+1}{2}- \frac{\lfloor\frac{4(k+1)}{3}\rfloor}{2}, & \mbox{if $n \leq \frac{f_{k+2}}{2}$},\\
2n +\frac{f_k+1}{2} - \frac{\lfloor\frac{4(k+1)}{3}\rfloor}{2} - \lceil \frac{f_{k+2}-1}{2}\rceil, & \mbox{if $n > \frac{f_{k+2}}{2}$}.
\end{cases}
\end{equation*} 
\end{theorem}

The result in Theorem 2.5 motivates us to introduce the notion of the loop sequence of a graph as follows:

\begin{definition}{\rm 
The \textit{loop sequence} is the sequence of the number of loops of the vertices, $v_i \in V(G)$, $1\leq i \leq n$. For a simple graph $G$ of order $n$, the loop sequence is $(l(v_i)=0:1\leq i\leq n)$.
}\end{definition}

For $n\geq 1$, Theorem 2.5 results in an integer list given by $\mathcal{L}_n = (0,1,2,3,4,5,7,8,\\ 9,10,12,14,15,16,17,18,19,21,23,25,26,\ldots, l(v_{2^n-1,1}))$. Note that certain integers such as $6,11,13,20,22,\ldots$ are excluded from the list.

\begin{theorem}\label{Thm-2.6}
For $n\geq 1$, the loop sequence of the corresponding Fibonacci-sum set-graph $G^F_{A^{(n)}}$ contains at least one of each entry in $\mathcal{L}_n$.
\end{theorem}
\begin{proof}
We have to prove that for all $i\in \mathcal{L}_n$, there exists at least one vertex $v_{s,i} \in V(G^F_{A^{(n)}})$ with $l(v_{s,i})=i$.

Consider any integer $n\in \N$ and its corresponding Fibonacci-sum set-graph. All vertices $v_{1,i}$, $1 \leq i\leq n$ are loop-free. Hence, $l(v_{1,i})=0$, $1\leq i \leq n$. From Lemma \ref{Lem-2.4}, it follows that a unique vertex exists with loop number equal to $l(v_{2^n-1,1})$. 

For $n\geq 2$, all Fibonacci-sum set-graphs have the vertex $v_{2,1}\mapsto \{1,2\}$ which has a single loop. Therefore, at least the vertex $v_{2,1}$ exists with $l(v_{2,1})=1$. Similarly, for $n\geq 3$, all Fibonacci-sum set-graphs have the vertex $v_{3,1}\mapsto \{1,2,3\}$ which has a double loop. Therefore, at least the vertex $v_{3,1}$ exists with $l(v_{3,1})=2$. Similarly, for $n\geq 4$, all Fibonacci-sum set-graphs have the vertex $v_{4,1}\mapsto \{1,2,3,4\}$ which has a triple loop. Therefore, at least the vertex $v_{4,1}$ exists with $l(v_{4,1})=3$.

Therefore, by induction read together with the well-ordering property of the set of positive integers, the well-ordering of the power set $\mathcal{P}(A^{(n)})-\emptyset$ and the well-defined set $\mathcal{L}_n$, the result follows in general.
\end{proof}

\begin{theorem}\label{Thm-2.7}
For $n\geq 1$ all vertex degrees in the corresponding Fibonacci-sum set-graph are even.
\end{theorem}
\begin{proof}
As the loops account for an even count to a vertex degree, it is sufficient to prove the result for the partial vertex degrees accounted for by incident edges only. 

The result can easily be verified for $n=1,2$. Assume the result holds for $n=k\geq 3$. Following the partial construction as explained in the proof of Theorem \ref{Thm-2.1}, it follows that the vertex degrees in the disjoint union, $G_1\cup G_2$ are even. Edges between $G_1$ and $G_2$ are added as follows. Each edge $v_{s,i}v_{t,j}$ in $G_1$ corresponds to the edges $v_{s,i}v'_{t,j}$ and $v_{t,j}v'_{s,i}$ respectively. Hence, this partial construction has even vertex degrees.

Finally, for $n\geq 2$, any element $j\in A^{(n)}$ is contained in $2^{n-1}$ distinct subsets of the power set $\mathscr{P}(A^{(n)})-\emptyset$. Hence, the number of all possible pairs $(i,j)$ such that $i+j$ is a Fibonacci number, is even. Thus, adding all additional edges to complete the construction of $G^F_{A^{(n)}}$ results in all vertex degrees being even.
\end{proof}

\begin{corollary}\label{Cor-2.8}
For $n\geq 1$ the Fibonacci-sum set-graph is Eulerian.
\end{corollary}
\begin{proof}
This result is a direct consequence of Theorem \ref{Thm-2.7}, since $G^F_{A^{(n)}}$ has no vertex with odd degree.
\end{proof}

\begin{definition}\label{Defn-2.3}{\rm 
The \textit{popped graph} of a graph $G$, denoted by $G^{^\backepsilon}$, is the graph obtained by deleting all loops and all multiple edges except one (to retain adjacency) from $G$. 
}\end{definition}

\noi Figure \ref{fig:fig-3} depicts $G^{F^{^\backepsilon}}_{A^{(3)}}$.

\begin{figure}[h!]
\centering
\begin{tikzpicture}[auto,node distance=1.75cm,
thick,main node/.style={circle,draw,font=\sffamily\Large\bfseries}]
\vertex (v1) at (51.42:3) [fill,label=right:$v_{3,1}$]{};
\vertex (v2) at (0:3) [fill,label=right:$v_{1,1}$]{};
\vertex (v3) at (308.57:3) [fill,label=below:$v_{1,2}$]{};
\vertex (v4) at (257.14:3) [fill,label=below:$v_{1,3}$]{};
\vertex (v5) at (205.71:3) [fill,label=left:$v_{2,1}$]{};
\vertex (v6) at (154.28:3) [fill,label=left:$v_{2,2}$]{};
\vertex (v7) at (102.85:3) [fill,label=above:$v_{2,3}$]{};
\path 
(v1) edge (v2)
(v1) edge (v3)
(v1) edge (v4)
(v1) edge (v5)
(v1) edge (v6)
(v1) edge (v7)
(v2) edge (v3)
(v2) edge (v5)
(v2) edge (v6)
(v2) edge (v7)
(v3) edge (v4)
(v3) edge (v5)
(v3) edge (v6)
(v3) edge (v7)
(v4) edge (v5)
(v4) edge (v7)
(v5) edge (v6)
(v5) edge (v7)
(v6) edge (v7)
;
\end{tikzpicture}
\caption{Popped Fibonacci-sum set-graph $G^{F^{^\backepsilon}}_{A^{(3)}}$.}\label{fig:fig-3}
\end{figure}
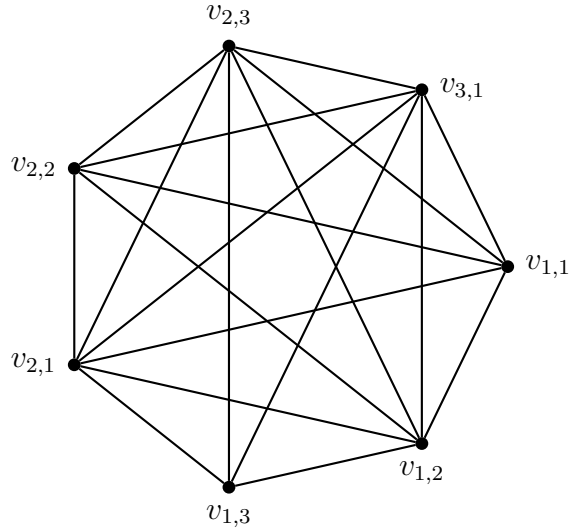 

\begin{proposition}\label{Prop-2.9}
For $n\geq 1$, the popped Fibonacci-sum set-graph has $\varepsilon(G^{F^{^\backepsilon}}_{A^{(n)}}) \leq (2^n-2) + \sum\limits_{u\in N(v_{2^n-1,1})}\epsilon(v_{2^n-1,1},u)$.
\end{proposition}
\begin{proof}
Clearly, $\varepsilon(G^F_{A^{(n)}})=\varepsilon(G^{F^{^\backepsilon}}_{A^{(n)}})=2^n-1$. For $n=1$, the equality holds and its verification is straight forward.

For $n\geq 2$, it can easily be verified that vertex $v_{2^n-1,1}$ is adjacent to all vertices and hence in the popped graph $G^{F^{^\backepsilon}}_{A^{(n)}}$, we have $d(v_{2^n-1,1})=2^n-2$. With respect to the second term $\sum\limits_{u\in N(v_{2^n-1,1})}\epsilon(v_{2^n-1,1},u)$ mathematical induction will be used. For $n=2$ the popped graph $G^{F^{^\backepsilon}}_{A^{(2)}} \cong K_3$. Hence, $\varepsilon(K_3) = 3 < 5 = (2^2-2) + \sum\limits_{u\in N(v_{2^2-1,1})}\epsilon(v_{2^2-1,1},u)$.

For $n=3$, Figure \ref{fig:fig-2} shows that $\sum\limits_{u\in N(v_{2^3-1,1})}\epsilon(v_{2^3-1,1},u) =15$, whilst Figure \ref{fig:fig-3} shows that $\varepsilon(G^{F^{^\backepsilon}}_{A^{(3)}} -v_{3,1})= 13$. Therefore, $\varepsilon(G^{F^{^\backepsilon}}_{A^{(3)}}) < (2^3-2) + \sum\limits_{u\in N(v_{2^3-1,1})}\epsilon(v_{2^3-1,1},u)$.

Now, consider the case $n=k$. Note that for any distinct pair of vertices $v_{s,i} \neq v_{2^k-1,1}$, $v_{t,j} \neq v_{2^k-1,1}$ each edge $v_{s,i}v_{t,j}$ which exists represents that, say $i' \in A^{(n)}_{s,i}$ and $j' \in A^{(n)}_{t,j}$, such that $i'+j'$ is a Fibonacci number. Since both $i',j' \in A^{(n)}_{2^k-1,1}$, two edges $v_{2^k-1,1}v_{s,i}$ and $v_{2^k-1,1}v_{t,j}$ must exist because of the existence of edge $v_{s,i}v_{t,j}$. Therefore, the general result $\varepsilon(G^{F^{^\backepsilon}}_{A^{(n)}}) < (2^n-2) + \sum\limits_{u\in N(v_{2^n-1,1})}\epsilon(v_{2^n-1,1},u)$, holds for $n\geq 2$ by induction.
\end{proof}

\section{Conclusion}


We know that a Fibonacci-sum graph $G^F_n,\ n\geq 2$ is bipartite and hence $\chi(G^F_n)=2$, where $n\geq 2$ (see \cite{AGL}). From Corollary \ref{Cor-2.3} and the fact that a Fibonacci-sum set-graph has an odd number of vertices imply that for $n\geq 2$ an odd spanning cycle exists and therefore a Fibonacci-sum set-graph is not $2$-colourable. We observe that the respective rainbow neighbourhood numbers are  $r^+_\chi(G^F_n) = r^-_\chi(G^F_n)= 2^n-1$ (see \cite{KS1}). An important invariant for further investigations and understanding of colouring is the clique number $\omega(G)$. Determining this invariant for Fibonacci-sum set-graphs is still open.

An \textit{eared clique} is an induced complete subgraph of a graph $G$ with the loops of each vertex and multiple edges which exist in $G$. The order of a largest eared clique of a graph $G$ is called the \textit{eared clique number} of $G$ and is denoted by $\omega_e(G)$.  Clearly $\omega_e(G) = \omega(G^{^\backepsilon})$. When the context is clear the invariant $\omega(G^{^\backepsilon})$ will be used.

\begin{problem}{\rm 
If possible, determine the eared clique number of the popped Fibonacci-sum set-graph $\omega(G^{F^{^\backepsilon}}_{A^{(n)}})$.
}\end{problem}

\begin{problem}{\rm
If possible, determine the chromatic number of the popped Fibonacci-sum set-graph $G^{F^{^\backepsilon}}_{A^{(n)}}$.
}\end{problem}

\begin{problem}{\rm
If possible, improve on the strict inequality of Proposition \ref{Prop-2.9} for $n\geq 2$.
}\end{problem}

Beyond the limited study reported in this paper, the new notion of Fibonacci-sum set-graphs offers a wide scope for further research. Similar integer sequence-sum set-graphs such as the Lucas-sum set-graphs offer an interesting new direction for research.

In particular, the idea of the set-graph of a graph $G$ of order $n$ comes to the fore. Let $A^{(n)}=V(G)=\{v_i:1\leq i \leq n\}$ and define the set-graph vertices as is understood from Definition \ref{Defn-2.1}. The set-graph of a graph $G$, denoted by $G_{V(G)}$, has edges of the form $\{v_{s,i}v_{t,j}\}$ for all $(v_{i'},v_{j'}),v_{i'}\in A^{(n)}_{s,i}$  and $v_{j'}\in A^{(n)}_{t,j}, v_{i'} \neq v_{j'}$ and edge $v_{i'}v_{j'} \in E(G)\}$. So, loops and multiple edges are permitted. The study of the set-graph of a graph is open.

\end{document}